\documentclass[11pt,reqno,a4paper]{amsart}

\usepackage{euscript}
\usepackage{amsmath,amssymb,amsfonts}
\usepackage{hyperref}

\usepackage{hyperref}
\RequirePackage[dvipsnames]{xcolor} 
\definecolor{halfgray}{gray}{0.55} 
\definecolor{webgreen}{rgb}{0,0.5,0}
\definecolor{webbrown}{rgb}{.6,0,0} \hypersetup{%
  colorlinks=true, linktocpage=true, pdfstartpage=3,
  pdfstartview=FitV,%
  breaklinks=true, pdfpagemode=UseNone, pageanchor=true,
  pdfpagemode=UseOutlines,%
  plainpages=false, bookmarksnumbered, bookmarksopen=true,
  bookmarksopenlevel=1,%
  hypertexnames=true,
  pdfhighlight=/O,
  urlcolor=webbrown, linkcolor=RoyalBlue,
  citecolor=webgreen, 
  pdftitle={},%
  pdfauthor={},%
  pdfsubject={2000 MAthematical Subject Classification: Primary:},%
  pdfkeywords={},%
  pdfcreator={pdfLaTeX},%
  pdfproducer={LaTeX with hyperref}%
}

\newtheorem{theorem}{Theorem}

\newtheorem{lemma}{Lemma}
\newtheorem{Claim}{Claim}

\newtheorem{definition}{Definition}
\newtheorem{remark}{Remark}

\renewcommand{\epsilon}{\varepsilon}

\DeclareMathOperator{\Ker}{Ker}

\def\Id{\text{\rm Id}}

\def\N{\mathbb{N}}

\def\R{\mathbb{R}}

\begin{document}

\title[Shadowing and hyperbolicity for delay equations]{Shadowing, Hyers--Ulam stability and hyperbolicity for nonautonomous linear delay differential equations}

\author[L. Backes]{Lucas Backes}
\address{\noindent Departamento de Matem\'atica, Universidade Federal do Rio Grande do Sul, Av. Bento Gon\c{c}alves 9500, CEP 91509-900, Porto Alegre, RS, Brazil.}
\email{lucas.backes@ufrgs.br}

\author[D. Dragi\v cevi\'{c}]{Davor Dragi\v cevi\'c}
\address{Faculty of Mathematics, University of Rijeka, Radmile Matej\v ci\' c 2, 51000 Rijeka, Croatia}
\email{ddragicevic@math.uniri.hr}

\author[M. Pituk]{Mih\'aly Pituk}
\address{Department of Mathematics, University of Pannonia, Egyetem \'ut 10, 8200 Veszpr{\'e}m, Hungary;
HUN-REN ELTE Numerical Analysis and Large Networks Research Group, Budapest, Hungary}
\email{pituk.mihaly@mik.uni-pannon.hu}



\begin{abstract}
It is known that hyperbolic 
non\-autonomous linear delay differential equations in a finite dimensional space are Hyers--Ulam stable and hence shadowable. The converse result is available only in the special case of autonomous and periodic linear delay differential equations with a simple spectrum. In this paper, we prove the converse and hence the equivalence of all three notions in the title for a general class of nonautonomous linear delay differential equations with uniformly bounded coefficients. The importance of the boundedness assumption is shown by an example.
\end{abstract}

\maketitle

\section{Introduction}

In this paper, we are interested in the relationship between the three notions in the title for a general class of nonautonomous linear delay differential equations in~$\mathbb R^d$. The shadowing property and the closely related concept of Hyers--Ulam stability of a dynamical system requires that in a neighborhood of an approximate solution there exists a true solution. The nonautonomous counterpart of hyperbolicity is the notion of an exponential dichotomy. Roughly speaking, a linear delay equation admits an exponential dichotomy if, at each time~$t$, the phase space can be decomposed into the direct sum two closed subspaces, called the stable and unstable subspace at~$t$, 
respectively, 
such that along the stable subspaces the associated evolution family is exponentially contractive and along the unstable subspaces it is exponentially expanding. For the precise notions in the context of delay equations, see Definitions~\ref{shadow}--\ref{ed} in Sec.~\ref{sec: definitions}. 

Recently, it has been shown that hyperbolic nonautonomous linear delay equations are Hyers--Ulam stable and shadowable~\cite{BDPS}, while the converse result has been proved only for autonomous and periodic equations under additional spectral conditions~\cite{BV}. It should be emphasized that delay differential equations form a special case of infinite dimensional systems and for infinite dimensional systems, even in the simplest case of autonomous linear dynamics, the shadowing property alone does not imply hyperbolicity~\cite[Theorem B and Remark 14]{Pujals}. It is therefore a nontrivial question whether the shadowing property implies hyperbolicity for nonautonomous linear delay differential equations. In Sec.~\ref{sec: definitions}, we give an affirmative answer to this question for a general class of linear delay equations with uniformly bounded coefficients. As we shall see in the proof, this follows from the eventual compactness of the solution operator and from the fact that shadowable linear delay differential equations have the so-called Perron property. The Perron property guarantees that the corresponding nonhomogeneous equation has at least one bounded solution for every bounded nonhomogeneity. For other variants of the Perron property and their consequences for the stability of delay differential equations, see~\cite{BerBrav}, \cite{BerDib}, \cite{DibZaf} and the references therein. The Perron property, combined with the compactness of the solution operator, will be used to show that the stable subspaces are closed and have constant finite codimension. This implication is a consequence of a result due to Sch\"affer~\cite{Schaffer} about regular covariant subspaces corresponding to nonautonomous difference equations with compact coefficient operators in a Banach space. The importance of Sch\"affer's result has been also recognized and used by Pecelli~\cite{Pecelli2} in his study of admissibility of various pairs of function spaces for delay equations although he has not considered the case when the input and output spaces consist of bounded and continuous functions 
which is needed for our purposes.

It follows from the above arguments that the
stable subspaces of 
a shadowable linear delay equation are complemented and their complements, the corresponding unstable sets, have constant finite dimension. Solutions starting from the unstable sets can be continued backwards. At this point it is worth mentioning that Kurzweil~\cite{Kurz} obtained a similar qualitative result. In \cite[Theorems~4.1 and~4.2]{Kurz}, he proved that for nonautonomous linear delay differential equations with integrally bounded coefficients the space of those solutions which are (exponentially) bounded as $t\to-\infty$ is finite dimensional and he provided estimates for its dimension.

 After establishing the direct sum decomposition of the phase space into stable and unstable subspaces, we will use the Perron property again and a technique from the admissibility theory (\cite{Dal}, \cite{MasSchaf}, \cite{MRS}, \cite{Pecelli2}) to show the required exponential estimates of the exponential dichotomy. For a detailed exposition of the admissibility theory, see~\cite{BDV} and references therein.

The paper is organized as follows. In Sec.~\ref{sec: definitions}, we introduce the definitions and we formulate our main results. A brief comparison with existing results and the importance of the boundedness assumption are discussed as well. The proof of the main theorem is presented in Sec.~\ref{sec: pr}.

\section{Main result}\label{sec: definitions}

Given $r\geq0$, let $C=C([-r,0],\mathbb R^d)$ be the Banach space of all continuous maps $\phi \colon [-r, 0] \to \R^d$ equipped with the supremum norm, $\|\phi \|:=\sup_{\theta \in [-r, 0]} |\phi(\theta)|$ for $\phi\in C$, where $|\cdot |$ is any norm on $\R^d$. As usual, the symbol $\mathcal L(C,\R^d)$ denotes the space of bounded linear operators from~$C$ into~$\R^d$ equipped with the operator norm. 

Consider the nonautonomous linear delay differential equation
\begin{equation}\label{LDE}
x'(t)=L(t)x_t,
\end{equation}
where $L\colon[0,\infty)\to \mathcal L(C,\R^d)$ is continuous and $x_t\in C$ is defined by $x_t(\theta)=x(t+\theta)$ for $\theta \in [-r, 0]$. Given $s\geq0$, by a \emph{solution of~\eqref{LDE} on $[s,\infty)$}, we mean a continuous function $x\colon [s-r, \infty)\to \R^d$ which is differentiable on $[s, \infty)$ and satisfies~\eqref{LDE} for all $t\ge s$. (By the derivative at $t=s$, we mean the right-hand derivative.) It is well-known~(see \cite[Chap.~6]{Hale}) that, for every $s\geq0$ and $\phi\in C$, Eq.~\eqref{LDE} has a unique solution~$x$ on $[s,\infty)$ with initial value $x_s=\phi$. For $t\geq s\geq0$, the \emph{solution operator} $T(t,s)\colon C\to C$ is defined by 
$T(t,s)\phi=x_t$, where $x$ is the unique solution of~\eqref{LDE} 
with $x_s=\phi$.

In the following definitions, we recall the notions of shadowing, Hyers--Ulam stability and exponential dichotomy for Eq.~\eqref{LDE}.
\begin{definition}\label{shadow}
{\rm
We say that Eq.~\eqref{LDE} is \emph{shadowable} if for each $\epsilon >0$ there exists $\delta>0$ with the following property: for each continuous function $y\colon [-r, \infty)\to \mathbb R^d$ which is continuously differentiable on $[0, \infty)$ and satisfies
\[
\sup_{t\ge 0}|y'(t)-L(t)y_t| \le \delta, 
\]
there exists a solution 
$x$ of~\eqref{LDE} on $[0,\infty)$ such that  
\[
\sup_{t\ge 0} \|x_t-y_t\| \le \epsilon.
\]
}
\end{definition}

\begin{definition}\label{hus}
{\rm
We say that Eq.~\eqref{LDE} is \emph{Hyers--Ulam stable} if there exists $\kappa>0$ such that for each continuous function $y\colon [-r, \infty)\to \mathbb R^d$ which is continuously differentiable on $[0, \infty)$ and satisfies
\[
\sup_{t\ge 0}|y'(t)-L(t)y_t| \le \delta \qquad\text{for some $\delta>0$},
\]
 there exists a solution 
$x$ of~\eqref{LDE} on $[0,\infty)$ such that 
\[
\sup_{t\ge 0} \|x_t-y_t\| \le \kappa\delta.
\]
}
\end{definition}

\begin{definition}\label{ed}
{\rm
We say that Eq.~\eqref{LDE} admits an \emph{exponential dichotomy} if there exist a family of projections $(P(t))_{t\geq0}$ on~$C$ and constants $D, \lambda >0$ with the following properties:
\begin{itemize}
\item for $t\ge s\ge 0$, 
\begin{equation}\label{pro}
P(t)T(t,s)=T(t,s)P(s),
\end{equation}
and $T(t,s)\rvert_{\Ker P(s)} \colon \Ker P(s)\to \Ker P(t)$ is onto and invertible;
\item for $t\ge s\ge 0$, 
\begin{equation}\label{eq: def est stable}
\|T(t,s)P(s)\| \le De^{-\lambda (t-s)};
\end{equation}
\item for $0\le t\le s$,
\begin{equation}\label{eq: def est unstable}
\|T(t,s)Q(s)\| \le De^{-\lambda (s-t)},
\end{equation}
where $Q(s)=\Id-P(s)$ for $s\geq0$ and $T(t,s):=\left (T(s,t)\rvert_{\Ker P(t)} \right )^{-1}$ for $0\leq t\leq s$.
\end{itemize}
}
\end{definition}

It follows immediately from the definitions that if Eq.~\eqref{LDE} is Hyers--Ulam stable, then it is shadowable. The following result is a corollary of a more general theorem on weighted shadowing from~\cite{BDPS}. 

\begin{theorem}\label{T1}
{\rm \cite[Theorem~2.3 ($\gamma=0$, $f\equiv0$)]{BDPS}}
If Eq.~\eqref{LDE} has an exponential dichotomy, then it is Hyers--Ulam stable and hence shadowable. 
\end{theorem}

The purpose of this paper is to show that if the linear operators $L(t)\colon C\to\mathbb R^d$, $t\geq0$, 
are uniformly bounded, then the converse of Theorem~\ref{T1} is also true. The importance of the uniform boundedness of the coefficients will be illustrated by an example. Our main result is the following theorem.
\begin{theorem}\label{T2}
Suppose that
\begin{equation}\label{bc}
M:=\sup_{t\ge 0} \|L(t)\|<\infty.
\end{equation}
If Eq.~\eqref{LDE} is shadowable, then it has an exponential dichotomy.
\end{theorem}

The proof of Theorem~\ref{T2} will be given in Sec.~\ref{sec: pr}.

As a corollary of Theorems~\ref{T1} and~\ref{T2}, we obtain that under condition~\eqref{bc}  all three notions in the title are equivalent for Eq.~\eqref{LDE}.
\begin{theorem}\label{T3}
If~\eqref{bc} holds, then the  following statements are equivalent:
\begin{enumerate}
\item[(i)] Eq.~\eqref{LDE} is Hyers--Ulam stable;
\item[(ii)] Eq.~\eqref{LDE} is shadowable;
\item[(iii)] Eq.~\eqref{LDE} admits an exponential dichotomy.
\end{enumerate}
\end{theorem}

\begin{remark}\label{R1}
{\rm
Theorem~\ref{T3} is a significant improvement of the recent results by Barreira and Valls~\cite{BV}, where the equivalence of Hyers--Ulam stability and the existence of an exponential dichotomy 
has been proved only in the rather special case of autonomous and periodic linear delay differential equations under the additional assumption that their spectrum is ``simple" (see \cite[Theorems~4.3 and~4.7]{BV}).
}
\end{remark}

\begin{remark}\label{R2}
{\rm The boundedness assumption~\eqref{bc} in Theorems~\ref{T2} and~\ref{T3} cannot be omitted. This can be shown by the following example taken from~\cite{Dal}.
Let  $v\colon[0,\infty)\rightarrow(0,\infty)$ be a positive, continuously differentiable function satisfying
\begin{gather}
\int_0^t v(s)\,ds\leq v(t),\qquad t\geq0,\label{vpr1}\\
v(t)\to\infty,\qquad t\to\infty,\label{vpr2}
\end{gather}
and
\begin{equation}\label{vpr3}
\frac{v(n-\alpha_n)}{v(n)}>n,\qquad n\in\mathbb N,
\end{equation}
where $\alpha=(\alpha_n)_{n\in\mathbb N}$ is a sequence of positive numbers such that $\alpha_n\to0$ as $n\to\infty$. For an explicitly given~$v$ and~$\alpha$ with the above
properties, see~\cite[p.~131]{Dal}. Consider the scalar ordinary differential equation
\begin{equation}\label{sode}
	x'(t)=a(t)x(t),
\end{equation}
where
\begin{equation*}
	a(t)=-\frac{v'(t)}{v(t)},\qquad t\geq0.
\end{equation*}
Eq.~\eqref{sode} is a special case of~\eqref{LDE} when $r=0$, $d=1$ and $L(t)\phi=a(t)\phi(0)$ for $t\geq0$ and $\phi\in C$. The solutions of Eq.~\eqref{sode} have the form
\begin{equation}\label{solrep}
	x(t)=\frac{v(s)}{v(t)}x(s),\qquad t,s\in[0,\infty).
\end{equation}
Hence
\[
[T(t,s)\phi](0)=\frac{v(s)}{v(t)}\phi(0),\qquad t\geq s\geq 0,\quad
\phi\in C,
\]
which implies that
\begin{equation}\label{normeq}
\|T(t,s)\|=\frac{v(s)}{v(t)},\qquad t\geq s\geq0.	
\end{equation}
From this and~\eqref{vpr1}, we find for $t\geq0$,
\[
\int_0^t\|T(t,s)\|\,ds\leq\int_0^t\frac{v(s)}{v(t)}\,ds\leq1.
\]
This shows that the assumptions of \cite[Theorem~2.2]{BDPS} are satisfied with $f\equiv0$, $w\equiv1$, $P(t)\equiv\Id$, $Q(t)\equiv0$, $M=1$ and $c=0$. By the application of \cite[Theorem~2.2]{BDPS}, we conclude that Eq.~\eqref{sode} is Hyers--Ulam stable and hence shadowable. We will show that Eq.~\eqref{sode} has no exponential dichotomy. Suppose, for the sake of contradiction, that Eq.~\eqref{sode} has an exponential dichotomy. It follows from the definition that if Eq.~\eqref{LDE} has an exponential dichotomy and all solutions of~\eqref{LDE} are bounded, then $Q(s)=0$ and hence $P(s)=\Id$ for all $s\geq0$, which implies that the exponential dichotomy is an exponential contraction, i.e
\begin{equation}\label{expcontr}
	\|T(t,s)\| \le De^{-\lambda (t-s)},\qquad t\geq s\geq0.
\end{equation}
Indeed, if $Q(s)\phi\neq0$ for some $s\geq0$ and $\phi\in C$, then~\eqref{eq: def est unstable} implies that the norm of the solution $x_t=T(t,s)\phi$ of~Eq.~\eqref{LDE} tends to infinity exponentially as $t\to\infty$, which contradicts the boundedness of~$x$. 
Since Eq.~\eqref{sode} has an exponential dichotomy and ~\eqref{vpr2} and~\eqref{solrep} imply that all solutions of~\eqref{sode} are bounded, we have that~\eqref{expcontr} holds. 
From~\eqref{normeq} and~\eqref{expcontr}, we obtain for $n\in\mathbb N$,
\[
\frac{v(n-\alpha_n)}{v(n)}=\|T(n,n-\alpha_n)\|\leq De^{-\lambda\alpha_n}.
\]
From this, letting $n\to\infty$ and taking into account that $\alpha_n\to0$ as $n\to\infty$, we conclude that
\[
\limsup_{n\to\infty}\frac{v(n-\alpha_n)}{v(n)}\leq D.
\]
On the other hand,~\eqref{vpr3} implies that $\lim_{n\to\infty}\frac{v(n-\alpha_n)}{v(n)}=\infty$, which yields a contradiction.
Thus, Eq.~\eqref{sode} is Hyers--Ulam stable and shadowable, but it has no exponentially dichotomy.
}
\end{remark}

\section{Proof of the main theorem}\label{sec: pr}
In this section, we a give a proof of Theorem \ref{T2}.
The key arguments are the following:
\begin{itemize}
\vskip5pt
 \item The shadowing property of Eq.~\eqref{LDE} implies the Perron type property. Name\-ly, the nonhomogeneous equation associated with Eq.~\eqref{LDE} has at least one bounded solution for every bounded and continuous nonhomogeneity.
 \vskip5pt
 \item The Perron property, combined with the eventual compactness of the solution operator and Sch\"affer's result about regular covariant sequences corresponding to compact linear operators in a Banach space, implies that the stable subspace of Eq.~\eqref{LDE} is closed and has finite codimension.
 \vskip5pt
 \item  The fact that the stable subspace is complemented yields a direct sum decomposition of the phase space into stable and unstable subspaces at each time instant in a standard manner.
 \vskip5pt
 \item  The required exponential estimates along the stable and unstable directions are obtained by appropriate adaptation of a technique from the admissibility theory of ordinary differential equations to delay equations.
\end{itemize}

\noindent
Now we present the details.
\begin{proof}[Proof of Theorem~\ref{T2}]
 Suppose that Eq.~\eqref{LDE} is shadowable. We will show that Eq.~\eqref{LDE} admits an exponential dichotomy. For better transparency,  
we split the proof into a series of auxiliary results. 

\begin{Claim}\label{LEM1}
Eq.~\eqref{LDE} has the following \emph{Perron type property}:
for each bounded and continuous function $z\colon [0, \infty)\to \R^d$,  
there exists a bounded and continuous function $x\colon [-r, \infty)\to \R^d$ which is differentiable on $[0, \infty)$ 
and satisfies
\begin{equation}\label{adm}
x'(t)=L(t)x_t+z(t), \qquad t\ge 0.
\end{equation}
\end{Claim}
\begin{proof}[Proof of Claim~\ref{LEM1}]
Let $z\colon [0, \infty)\to \R^d$ be an arbitrary bounded and continuous function. If $z(t)=0$ identically for $t\geq0$, then~\eqref{adm} is trivially satisfied with $x(t)=0$ for $t\geq-r$. Now suppose that $z(t)\neq0$ for some $t\geq0$ so that $\|z\|_\infty:=\sup_{t\ge 0}|z(t)|>0$.
Choose a constant $\delta >0$ corresponding to the choice of $\epsilon=1$ in Definition~\ref{shadow}. Take an arbitrary solution $y\colon [-r, \infty)\to \R^d$ of the nonhomogeneous equation
\[
y'(t)=L(t)y_t+\frac{\delta}{\|z\|_\infty}z(t), \qquad t\ge 0.
\]
(The unique solution~$y$ with initial value $y_0=0$ is sufficient for our purposes.) Since
\[
\sup_{t\ge 0}|y'(t)-L(t)y_t| \le \delta,
\]
according to Definition~\ref{shadow}, there exists a solution~$\tilde x$ 
 of Eq.~\eqref{LDE} on $[0,\infty)$ such that
\[
\sup_{t\ge -r}|\tilde x(t)-y(t)|=\sup_{t\ge 0}\|\tilde x_t-y_t\| \le 1.
\]
Define a continuous $x\colon [-r, \infty)\to \R^d$ by 
\[
x(t):=\frac{\|z\|_\infty}{\delta}(\,y(t)-\tilde x(t)\,), \qquad t\ge -r.
\]
It can be easily verified that $x$ satisfies~\eqref{adm} and 
\[
\sup_{t\ge -r}|x(t)|\le \frac{\|z\|_\infty}{\delta}<\infty.
\]
\end{proof}

For each $s\ge 0$, define
\[
\mathcal S(s)=\bigl \{ \phi\in C: \sup_{t\ge s}\|T(t,s)\phi \|<\infty \bigr \}.
\]
Clearly, $\mathcal S(s)$ is a subspace of $C$ which will be called the \emph{stable subspace of Eq.~\eqref{LDE} at time $s\geq0$}. 
\begin{Claim}\label{INV}
For each $t\ge s \ge 0$, we have that 
\[
[\,T(t,s)\,]^{-1}(\mathcal S(t))=\mathcal S(s).
\]
\end{Claim}
\begin{proof}[Proof of Claim~\ref{INV}]
Let $t\ge s \ge 0$. If $\phi \in \mathcal S(s)$, then, 
\[
\sup_{\tau\ge t}\|T(\tau,t)T(t, s)\phi \|=\sup_{\tau\ge t}\|T(\tau,s)\phi \|\le \sup_{\tau\ge s}\|T(\tau,s)\phi \|<\infty. 
\]
This shows that $T(t,s)\phi \in \mathcal S(t)$ and hence $\phi \in [\,T(t,s)\,]^{-1}(\mathcal S(t))$.

Now suppose that $\phi \in [\,T(t,s)\,]^{-1}(\mathcal S(t))$. Then, $T(t,s)\phi \in \mathcal S(t)$, which implies that
\begin{equation}\label{Tbound}
\sup_{\tau\ge t}\|T(\tau,s)\phi\|=\sup_{\tau\ge t}\|T(\tau,t)T(t,s)\phi\|<\infty.
\end{equation}
Let~$x$ denote the solution of Eq.~\eqref{LDE} with $x_s=\phi$. It is well-known (see, e.g., \cite[Sec.~2.2, Claim~2.1, p.~38]{Hale}) that the solution segment $x_\tau=T(\tau,s)\phi$ is a continuous function of $\tau\in[s,\infty)$ and hence it is bounded on the compact interval $[s,t]$. Therefore, $\sup_{s\le \tau\le t}\|T(\tau, s)\phi \|<\infty$, which, together with~\eqref{Tbound}, implies that
$\sup_{\tau\ge s}\|T(\tau,s)\phi \|<\infty$. Thus, $\phi \in \mathcal S(s)$.
\end{proof}

\begin{Claim}\label{ee}
For $t\ge s\ge 0$, we have the algebraic sum decomposition 
\begin{equation}\label{SPL}
C=T(t,s)C+\mathcal S(t).
\end{equation}
\end{Claim}

\begin{proof}[Proof of Claim~\ref{ee}]
It is sufficient to prove the claim for $s=0$. Indeed, assuming that the desired conclusion holds for $s=0$, we now take an arbitrary $s>0$. Then, for every $t\geq s$ and $\phi \in C$, there exist $\phi_1\in C$ and $\phi_2\in \mathcal S(t)$ such that $\phi=T(t,0)\phi_1+\phi_2$. Hence,
\[
\phi=T(t,0)\phi_1+\phi_2=T(t,s)T(s,0)\phi_1+\phi_2\in T(t,s)C+\mathcal S(t).
\]
Thus,~\eqref{SPL} holds.
Therefore, from now on, we suppose that $s=0$. Since $T(0,0)=\Id$, for every $\phi\in C$, we have
\[
\phi=T(0,0)\phi+0\in T(0,0)C+\mathcal S(0).
\]
Thus,~\eqref{SPL} holds for $t=s=0$. Now suppose that $t>0$ and let
 $\phi \in C$ be arbitrary. Choose a continuous function $\psi \colon [t-r,\infty) \to [0,\infty)$ with compact support contained in $[t, \infty)$ such that $\int_t^\infty \psi(\tau)\, d\tau=1$.  Let $u$ denote the unique solution of Eq.~\eqref{LDE} on $[t, \infty)$ with initial value $u_t=\phi$. Set
 \[
 \bar{\psi}(\tau):=\int_\tau^\infty \psi (\tau)\, d\tau\qquad\text{ for $\tau\geq t-r$}.
 \]
Finally, define $v\colon [t-r, \infty)\to \R^d$ and $z\colon [-r, \infty)\to \R^d$ by 
\[
v(\tau)=
\bar{\psi} (\tau)u(\tau)
 \qquad \text{for $\tau\ge t-r$}
\]
and
\[
z(\tau)=\begin{cases}
-\psi(\tau)u(\tau)+\bar{\psi}(\tau)L(\tau)u_\tau-L(\tau)(\bar{\psi}_\tau u_\tau) &\quad\text{for $\tau\ge t$}, \\
0&\quad\text{for $\tau\in[-r,t)$},
\end{cases}
\]
respectively.
Note that  $z$ is continuous on $[-r,\infty)$. The continuity of~$z$ at~$t$ follows from the relations $\psi(t)=0$, $\bar{\psi}(t)=1$, $\bar\psi_t\equiv1$ and $\bar{\psi}_t u_t=u_t$, which imply that $z(t)=0$. Since $\psi$ has compact support, we have that $\psi(\tau)=\bar\psi(\tau)=0$ for all large $\tau$. Hence, 
\[
\sup_{
\tau\ge t-r}|v(\tau)|<\infty\qquad\text{and}\qquad 
\sup_{\tau\ge -r} |z(\tau)|<\infty. 
\]
By Claim~\ref{LEM1}, there exists a   continuous function $x\colon [-r, \infty)\to \R^d$ which is differentiable on $[0, \infty)$ such that $\sup_{\tau\ge -r}|x(\tau)|<\infty$ and~\eqref{adm} holds. It is straightforward to show that 
\[
v'(\tau)=L(\tau)v_\tau+z(\tau), \qquad \tau\ge t.
\]
Hence, $x-v$ is a solution of Eq.~\eqref{LDE} on $[t, \infty)$ and thus
\begin{equation*}\label{g}
x_\tau-v_\tau=T(\tau,t)(x_t-v_t)=T(\tau,t)(x_t-\phi), \qquad \tau\ge t.
\end{equation*}
 From this and the boundedness of~$x$ and~$v$,  we conclude that $x_t-\phi \in \mathcal S(t)$. On the other hand, since $z\equiv 0$ on $[0, t]$, we have that $x_t=T(t, 0)x_0$. This implies that
\[
\phi=x_t+(\phi-x_t)=T(t,0)x_0+(\phi-x_t)\in T(t,0)C+ \mathcal S(t).
\]
Since $\phi \in C$ was arbitrary, we conclude that~\eqref{SPL} holds for $s=0$.
\end{proof}

\begin{Claim}\label{subcomplete}
For each $s\ge 0$, $\mathcal S(s)$ is the image of a  Banach space under the action of a bounded linear operator.
\end{Claim}

\begin{proof}[Proof of Claim~\ref{subcomplete}]
Fix $s\ge 0$ and 
let $\mathcal B$ denote the Banach space of all bounded and continuous functions $x\colon [s-r, \infty)\to \mathbb R^d$ equipped with the supremum norm,
\[
\|x\|_{\mathcal B}:=\sup_{t\ge s-r}|x(t)|<\infty,\qquad x\in\mathcal B.
\]
Let $\mathcal B'$ denote the set of all $x\in \mathcal B$ which are solutions of~\eqref{LDE} on $[s, \infty)$. We will show that $\mathcal B'$ is a closed subspace of $\mathcal B$. Let $(x^k)_{k\in \N}$ be a sequence in~$\mathcal B'$ such that $x^k \to y$ in $\mathcal B$ for some $y\in\mathcal B$. Then, $x^k\to y$ pointwise on $[s-r,\infty)$
and $x^k_s\to y_s$ in~$C$ as $k\to\infty$. From the last limit relation, it follows by the continuous dependence of the solutions on the initial data (see \cite[Sec.~6.1, Corollary~1.1, p.~143]{Hale}) that $x^k\to x$ locally uniformly on $[s-r,\infty)$ as $k\to\infty$, where~$x$ is the unique solution of Eq.~\eqref{LDE} on $[s,\infty)$ with initial value $x_s=y_s$. It follows from the uniqueness of the limit of $(x^k)_{k\in \N}$ that $y=x$ identically on $[s-r,\infty)$ and hence $y\in\mathcal B'$.
 This shows that $\mathcal B'$ is a closed subspace of $\mathcal B$ and hence it is a Banach space. Now define $\Phi \colon \mathcal B' \to C$ by $\Phi(x)=x_s$ for $x\in\mathcal B'$. Clearly, $\Phi$ is a bounded linear operator with 
$\| \Phi \| \le 1$ and $\Phi(\mathcal B')=\mathcal S(s)$.
\end{proof}

\begin{Claim}\label{covariance}
The stable subspace~$\mathcal S(0)$ of Eq.~\eqref{LDE} is closed and has finite codimension in~$C$.
\end{Claim}

As noted above, the proof of Claim~\ref{covariance} will be based on Sch\"affer's result about closed covariant sequences associated with compact linear operators in a Banach space~\cite{Schaffer}. Before we formulate Sch\"affer's result, let us recall the following notions.

Let~$X$ be a Banach space.
A subspace $S$ of~$X$ is called \emph{subcomplete in~$X$} if there exist a Banach space~$Z$ and a bounded linear operator $\Phi\colon Z\to X$ such that $\Phi(Z)=S$.

Let $A\colon\N_0\to\mathcal L(X)$ be an operator-valued map, where $\N_0$ and $\mathcal L(X)$ denote the set of nonnegative integers and the space of bounded linear operators in~$X$, respectively. For $n\geq m\geq0$, the corresponding \emph{transition operator} $U(n,m)\colon X\to X$ is defined by
\[
U(n,m)=A(n-1)A(n-2)\cdots A(m)\qquad\text{for $n>m\geq0$}
\]
and $U(m,m)=\Id$ for $m\geq0$. A sequence $Y=(Y(n))_{n\in\N_0}$ of subspaces in~$X$ is called a \emph{covariant sequence for~$A$} if 
\[
[\,A(n)\,]^{-1}(Y(n+1))=Y(n) \qquad\text{for all $n\in \N_0$}.
\]
A covariant sequence~$Y=(Y(n))_{n\in\N_0}$ for~$A$ is called \emph{algebraically regular} if
\[
U(n,0)X+Y(n)=X\qquad\text{for each $n\in \N_0$}.
\]
Finally, a covariant sequence~$Y=(Y(n))_{n\in\N_0}$ for~$A$ is called \emph{subcomplete} if the subspace $Y(n)$ is subcomplete in~$X$ for all $n\in\N_0$.

In the proof of Claim~\ref{covariance}, we will use the following result due to Sch\"affer.

\begin{lemma}\label{schaef}
{\rm (\cite[Lemma~3.4]{Schaffer})}
Let $X$ be a Banach space and $A\colon\N_0\to\mathcal L(X)$. Suppose that 
$Y=(Y(n))_{n\in\N_0}$ is a subcomplete algebraically regular covariant sequence for~$A$. If the transition operator $U(n,m)\colon X\to X$ is compact for some $n,m\in\N_0$, $n\geq m$, then the subspaces $Y(n)$, $n\in\N_0$, are closed and have constant finite codimension in~$X$.
\end{lemma}

We shall also need the following known result about the compactness of the solution operators of Eq.~\eqref{LDE}.

\begin{lemma}\label{compcrit}
{\rm (\cite[Chap.~3, Sec.~3.6]{Hale})}
$T(t,s)\colon C\to C$ is a compact operator whenever $s\ge 0$ and $t\ge s+r$.
\end{lemma}

Now we can give a short proof of Claim~\ref{covariance}.
\begin{proof}[Proof of Claim~\ref{covariance}]
Claims~\ref{INV}, \ref{ee} and~\ref{subcomplete} guarantee that the stable subspaces $Y(n):=\mathcal S(nr)\subset C$ of Eq.~\eqref{LDE} 
form a subcomplete algebraically regular covariant sequence for $A\colon\N_0\to\mathcal L(C)$ defined by
\[
A(n):=T((n+1)r, nr), \qquad n\in \N_0.
\]
According to Lemma~\ref{compcrit}, for each $n\in\N_0$, $A(n)\colon C\to C$ is compact. By the application of Lemma~\ref{schaef}, we conclude that $Y(0)=\mathcal S(0)$ is closed and has finite codimension in~$C$.
\end{proof}

By Claim~\ref{covariance}, the stable subspace~$\mathcal S(0)$ is closed and has finite codimension in~$C$. This implies that~$\mathcal S(0)$ is complemented in~$C$ (see, e.g., \cite[Lemma~4.21, p.~106]{Rud}). More precisely,  there exists a subspace~$\mathcal U$ of~$C$ such that $\dim\mathcal U=\operatorname{codim}\mathcal S(0)<\infty$ and
\begin{equation}\label{SPLIT}
C=\mathcal S(0)\oplus \mathcal U.
\end{equation}
\begin{Claim}\label{LEM2}
For each bounded and continuous function $z\colon [0, \infty)\to \R^d$,  
there exists a unique bounded and continuous function $x\colon [-r, \infty)\to \R^d$ with $x_0\in \mathcal U$ which is differentiable on $[0, \infty)$ and satisfies ~\eqref{adm}.   Moreover, there exists a constant $A>0$, independent of $z$, such that 
\begin{equation}\label{cgt}
\sup_{t\ge -r}|x(t)| \le A\sup_{t\ge 0}|z(t)|.
\end{equation}
\end{Claim}

\begin{proof}[Proof of Claim~\ref{LEM2}]
Let $z\colon [0, \infty)\to \R^d$ be an arbitrary bounded and continuous function.
By Claim~\ref{LEM1}, there exists a bounded and continuous function $\tilde x\colon [-r, \infty)\to \R^d$ which is differentiable on $[0, \infty)$ and satisfies
\[
\tilde x'(t)=L(t)\tilde x_t+z(t), \qquad t\ge 0.
\]
On the other hand, \eqref{SPLIT} implies the existence of~$\phi_1\in \mathcal S(0)$ and $\phi_2\in \mathcal U$ such that 
\[
\tilde x_0=\phi_1+\phi_2.
\]
Define $x\colon [-r, \infty)\to \R^d$ by 
\[
x(t)=\tilde x(t)-y(t), \qquad t\ge -r,
\]
where $y$ is a solution of Eq.~\eqref{LDE} such that $y_0=\phi_1$. Since $y_0=\phi_1\in\mathcal S(0)$, we have that $\sup_{t\geq -r} |y(t)|<+\infty$. Then, $x$ satisfies~\eqref{adm}, $x_0=\tilde x_0-\phi_1=\phi_2\in \mathcal U$ and $\sup_{t\ge -r}|x(t)|<+\infty$. 
We claim that $x$ with the desired properties is unique. Indeed, if $\bar x$ is an arbitrary function with the desired properties, then  $x-\bar x$ is a bounded solution of Eq.~\eqref{LDE}, which implies that $x_0-\bar x_0\in \mathcal U\cap \mathcal S(0)=\{0\}$. Thus, $x_0=\bar x_0$ and hence $x=\bar x$ identically on $[-r,\infty)$.

Finally, we show the existence of a constant $A>0$ such that~\eqref{cgt} holds. 
Let $\mathcal B_0$ and $\mathcal B_{-r}$ denote the Banach space of all bounded and continuous $\R^d$-valued functions defined on $[0,\infty)$ and
$[-r,\infty)$, respectively, equipped with the supremum norm.
For $z\in\mathcal B_0$, define 
$\mathcal F(z)=x$, where $x$ is the unique bounded solution of the nonhomogeneous equation~\eqref{adm} with $x_0\in\mathcal U$.
(The existence and uniqueness of~$x$ is guaranteed by the first part of the proof.) Evidendly, $\mathcal F(z)=x\in\mathcal B_{-r}$ for $z\in\mathcal B_0$ and $\mathcal F \colon \mathcal B_0\to \mathcal B_{-r}$ is a linear operator.
We will show that $\mathcal F$ is a closed operator. Let  $(z^k)_{k\in \N}$ be a sequence in $\mathcal B_0$ such that $z^k\to z$ for some $z\in\mathcal B_0$ and $x^k:=\mathcal F (z^k) \to x$ for some $x\in\mathcal B_{-r}$. Using a similar notation as in~\cite[Sec.~6.2]{Hale}, the symbol $x(0,\phi,z)$ will denote the unique solution of the nonhomogeneous equation~\eqref{adm} with initial value~$\phi$ at zero, where $\phi\in C$ and $z\in\mathcal B_0$. Thus, $x^k=x(0,\phi^k,z^k)$, where $\phi^k:=x^k_0$ for $k\in\mathbb N$. Since $z^k\to z$ in~$\mathcal B_0$ and the limit relation $x^k\to x$ in $\mathcal B_{-r}$ implies that $\phi^k=x^k_0\to x_0=:\phi$ in~$C$, by the continuous dependence of the solutions of~\eqref{adm} on the parameters (see, e.g., \cite[Sec.~6.1, Corollary~1.1, p.~143]{Hale}), we have for $t\geq -r$, 
\[
x(t)=\lim_{k\to\infty}x^k(t)=\lim_{k\to\infty}x(0,\phi^k,z^k)(t)=x(0,\phi,z)(t).
\]
Thus, $x$ is a solution of~\eqref{adm}. Since $x^k_0\in\mathcal U$ for $k\in\mathbb N$ and $\mathcal U$ is a finite-dimensional and hence closed subset of~$C$, we have that $x_0=\lim_{k\to\infty}x^k_0\in\mathcal U$. Therefore, $x\in\mathcal B_{-r}$ is a bounded solution of~\eqref{adm} with $x_0\in\mathcal U$. Hence $\mathcal F(z)=x$ which shows that 
$\mathcal F \colon \mathcal B_0\to \mathcal B_{-r}$
is a closed operator. According to the Closed Graph Theorem (see, e.g., \cite[Theorem~4.2-I, p.~181]{Taylor}), $\mathcal F$ is bounded, which implies that~\eqref{cgt} holds with $A=\|\mathcal F\|$, the operator norm of~$\mathcal F$.
\end{proof}
For $s\ge 0$, define
\[ \mathcal U(s)=T(s,0)\mathcal U
\]
so that $\mathcal U(0)=\mathcal U$.
It is easily seen that 
\begin{equation}\label{eq: invariance S U}
T(t,s)\mathcal S(s)\subset \mathcal S(t) \quad \text{and} \quad T(t,s)\mathcal U(s)=\mathcal U(t)
\end{equation}
whenever $t\geq s\geq0$.
\begin{Claim}\label{lem: invertibility}
For $t\ge s\geq0$, $T(t,s)\rvert_{\mathcal U(s)} \colon \mathcal U(s)\to \mathcal U(t)$ is invertible.
\end{Claim}
\begin{proof}[Proof of Claim~\ref{lem: invertibility}]
In view of~\eqref{eq: invariance S U}, we only need to show that the operator $T(t,s)\rvert_{\mathcal U(s)} \colon \mathcal U(s)\to \mathcal U(t)$ is injective. Let $t\geq s\geq0$ and $\phi \in \mathcal U(s)$ be such that $T(t,s)\phi=0$. Since $\phi \in \mathcal U(s)$, there exists $\bar{\phi}\in\mathcal U$ such that $\phi=T(s,0)\bar{\phi}$. Let $x\colon [-r, \infty)\to \R^d$ be the solution of~\eqref{LDE} such that $x_0=\bar{\phi}\in\mathcal U$. Since 
\[
x_t=T(t,0)\bar\phi=T(t,s)T(s,0)\bar\phi=T(t,s)\phi=0,
\]
we have that $x(\tau)=0$ for $\tau\geq t-r$, which implies that  $\sup_{\tau\ge -r}|x(\tau)|<\infty$. It follows from the uniqueness in Claim~\ref{LEM2}, applied for $z\equiv 0$, that $x\equiv 0$. This implies that $\bar{\phi}=\phi= 0$.
\end{proof}
\begin{Claim}\label{splitting}
For each $t\ge 0$, $C$ can be decomposed into the direct sum
\begin{equation}\label{SPLIT2} 
C=\mathcal S(t)\oplus \mathcal U(t).
\end{equation}
\end{Claim}
\begin{proof}[Proof of Claim~\ref{splitting}]
Since $\mathcal U(0)=\mathcal U$, if $t=0$, then the decomposition~\eqref{SPLIT2} follows immediately from~\eqref{SPLIT}. Now suppose that
$t>0$ and let $\phi \in C$ be arbitrary. Let $v\colon [t-r, \infty)\to \R^d$ and $z\colon [-r, \infty)\to \R^d$ be as in the proof of Claim~\ref{ee}. Since $z$ is continuous and $\sup_{s\ge -r} |z(s)|<\infty$, by Claim~\ref{LEM2} there exists a  unique continuous function $x\colon [-r, \infty)\to \R^d$ which is differentiable on $[0, \infty)$ such that $x_0\in \mathcal U$, $\sup_{t\ge -r}|x(t)|<\infty$ and~\eqref{adm} holds. By the same reasoning as in the proof of Claim~\ref{ee}, we have that $x_t-\phi\in \mathcal S(t)$. Moreover, $x_t=T(t,0)x_0\in \mathcal U(t)$. Consequently,
\[
\phi=(\phi-x_t)+x_t\in \mathcal S(t)+\mathcal U(t).
\]
Now suppose that $\phi \in \mathcal S(t)\cap \mathcal U(t)$. Then, there exists $\bar{\phi}\in \mathcal U$ such that $\phi=T(t,0)\bar{\phi}$. Consider the unique solution $x\colon [-r, \infty)\to \R^d$ of Eq.~\eqref{LDE} with $x_0=\bar{\phi}$. Then, $x$ satisfies~\eqref{adm} with $z\equiv 0$ and $x_0=\bar\phi\in \mathcal U$. Moreover, $x_t=T(t,0)\bar\phi=\phi\in\mathcal S(t)$ implies that $\sup_{t\ge -r}|x(t)|<\infty$. By the uniqueness in Claim~\ref{LEM2}, we  conclude that $x\equiv 0$. Therefore, $\bar\phi=0$ which implies that $\phi=0$.
\end{proof}

\begin{Claim}\label{C}
There exists $Q>0$ such that 
\[
\|T(t,s)\phi \| \le Q\|\phi \|, 
\]
for $t\ge s\ge 0$ and $\phi \in \mathcal S(s)$.
\end{Claim}
\begin{proof}[Proof of Claim~\ref{C}]
It is known (see, e.g., \cite[Sec.~6.1, Corollary~1.1, p.~143]{Hale}) that under condition~
~\eqref{bc} the evolution family $(T(t,s))_{t\geq s\geq0}$ is exponentially bounded, i.e. there exist $K, a>0$ such that 
\begin{equation}\label{bg}
\|T(t,s)\| \le Ke^{a(t-s)}, \quad t\ge s\ge 0.
\end{equation}
Fix $s\ge 0$, $\phi \in \mathcal S(s)$ and let $u\colon [s-r, \infty)\to \R^d$ be the solution of Eq.~\eqref{LDE} on $[s, \infty)$ such that $u_s=\phi$. Choose a continuously differentiable function
$\psi\colon[-r,\infty) \to[0,1]$ such that $\operatorname{supp}\psi\subset[s,\infty)$,
$\psi \equiv 1$ on $[s+1, \infty)$ and $|\psi'| \le 2$. 
Define $x\colon [-r, \infty)\to \R^d$ and $z\colon [0, \infty)\to \R^d$ by 
\[
x(t)=\psi(t) u(t),\qquad t\geq-r,
\]
and 
\[
z(t)=\psi'(t)u(t)+\psi(t)L(t)u_t-L(t)(\psi_t u_t),\qquad t\geq0.
\]
Clearly, $x$ and $z$ are continuous, $x$ is differentiable on $[0, \infty)$ and~\eqref{adm} holds.  Since $\phi \in \mathcal S(s)$, the solution~$u$ and hence~$x$ is bounded on $[-r,\infty)$. Moreover, $x_0=0\in \mathcal U$.
Note that $\psi\equiv0$ on $[-r,s]$ and hence $z\equiv0$ on $[0,s]$.
Furthermore, 
$ \psi\equiv1$ on $[s+1,\infty)$, which implies that $\psi'\equiv0$ on $[s+1,\infty)$, $\psi_t\equiv 1$ for $t\ge s+r+1$ and hence $z(t)=0$ for $t\geq s+r+1$. From this, using~\eqref{bc} and~\eqref{bg}, we find that 
\[
\begin{split}
\sup_{t\ge 0}|z(t)|&=\sup_{t\in[s,s+r+1]}|z(t)|\\
&\le 2\sup_{t\in [s, s+1]}|u(t)|+\sup_{t\in [s, s+r+1]}(\psi(t)|L(t)u_t|+|L(t)(\psi_t u_t)|) \\
&\le 2\sup_{t\in [s, s+1]}\|u_t\|+2M\sup_{t\in [s, s+r+1]}\| u_t\|\\
&= 2\sup_{t\in [s, s+1]}\|T(t,s)\phi\|+2M\sup_{t\in [s, s+r+1]}\|T(t,s)\phi\|\\
&\le 2Ke^a\| \phi\|+2MKe^{a(r+1)}\| \phi\|.
\end{split}
\]
From~\eqref{cgt}, taking into account that $\psi\equiv1$ on $[s+1,\infty)$, we conclude that 
\[
\sup_{t\ge s+1}|u(t)| \le \sup_{t\ge -r}|x(t)|  \le A\sup_{t\ge 0}|z(t)| \le 2A(Ke^a+MKe^{a(r+1)})\| \phi\|.
\]
Hence, 
\[
\|T(t,s)\phi \|=\|u_t\| \le 2A(Ke^a+MKe^{a(r+1)})\| \phi\|, \quad t\ge s+r+1.
\]
On the other hand, \eqref{bg} implies that 
\[
\|T(t,s)\phi \| \le Ke^{a(r+1)}\|\phi\|, \qquad t\in [s, s+r+1].
\]
Consequently, the conclusion of the claim holds with
\[
Q:=\max \{Ke^{a(r+1)}, 2A(Ke^a+MKe^{a(r+1)})\}>0.
\]
\end{proof}

\begin{Claim}\label{exp est stable}
There exist $D, \lambda >0$ such that 
\[
\|T(t,s)\phi \| \le De^{-\lambda (t-s)}\|\phi\|,
\]
for $t\ge s\ge 0$ and $\phi \in \mathcal S(s)$.
\end{Claim}

\begin{proof}[Proof of Claim~\ref{exp est stable}]
We claim that if
\begin{equation}\label{Nchoice}
N>eAQ(MQr+1)+r
\end{equation}
with~$A$ and~$Q$ as in Claims~\ref{LEM2} and~\ref{C}, respectively,
then for every $s\geq0$ and $\phi\in\mathcal S(s)$,
\begin{equation}\label{e}
\|T(t,s)\phi \| \le \frac 1 e \|\phi \|\qquad\text{for $t\geq s+N$}.
\end{equation}
Suppose, for the sake of contradiction, that \eqref{Nchoice} holds and there exist $s\geq0$ and $\phi\in\mathcal S(s)$ such that
\begin{equation}\label{low}
\|T(t_0,s)\phi \|>\frac 1 e \|\phi \|\qquad\text{for some $t_0\geq s+N$}.
\end{equation}
Let $u$ denote the solution of the homogeneous equation Eq.~\eqref{LDE} with initial value $u_s=\phi$ so that $u_{t_0}=T(t_0,s)\phi$. 
From~\eqref{low} and Claim~\ref{C}, we obtain 
\begin{equation}\label{lowup}
\frac 1 e \|\phi \|<\|u_{t_0}\|=\|T(t_0,\tau)u_\tau\|\leq Q\|u_\tau\|\qquad \text{for $\tau\in[s,t_0]$}.
\end{equation}
Since $t_0\geq s+N>s+r$, we can choose a continuous function $\psi\colon[s,t_0]\rightarrow[0,1]$ such that
\[
\operatorname{supp}\psi\subset[s,t_0-r]\qquad\text{and}\qquad\text{$\psi\equiv1$ on $[s+\epsilon,t_0-r-\epsilon]$},
\]
where $\epsilon\in(0,\frac{1}{2}(t_0-r-s))$.
In view of~\eqref{lowup}, we can define a function $x\colon[-r,\infty)\rightarrow\mathbb R^d$ by 
\[
x(t)=\begin{cases}
\chi(t)u(t)& \text{for $t\ge s$,}\\
0 & \text{for $t\in[-r,s)$,}
\end{cases}
\]
where $\chi\colon[-r,\infty)\rightarrow[0,\infty)$ is given by
\[
\chi(t)=\begin{cases}
0\qquad&\text{for $t\in[-r,s)$},\\
\displaystyle\int_s^t\psi(\tau)\|u_\tau\|^{-1}\,d\tau
\qquad&\text{for $t\in[s,t_0)$}, \\[10pt]
\displaystyle\int_s^{t_0}\psi(\tau)\|u_\tau\|^{-1}\,d\tau
\qquad&\text{for $t\in[t_0, \infty)$}.
\end{cases}
\]
Evidently, $x$ coincides with the trivial solution of~\eqref{LDE} on $[-r,s)$. Since $\psi\equiv0$ on $[t_0-r,\infty)$, we have that $x(t)=u(t)\chi(t_0-r)$ for $t\geq t_0-r$. Thus, on the interval $[t_0-r,\infty)$, the function~$x$ coincides with a constant multiple of the solution~$u$ of the homogeneous equation~\eqref{LDE}. This implies that~$x$ is a solution of~\eqref{LDE} on $[t_0,\infty)$. Clearly, $x(t)=\chi(t)u(t)=0$ for $t\in[s-r,s]$. Therefore, $x(t)=\chi(t)u(t)$ for $t\geq s-r$ and hence $x_t= u_t\chi_t$ for $t\geq s$. It follows by easy calculations that if $t\in[s,t_0)$, then 
\[
x'(t)=L(t)x_t+L(t)(\chi(t)u_t-u_t\chi_t)+\psi(t)u(t)\|u_t\|^{-1}.
\]
Therefore, $x$ is a solution of the nonhomogeneous equation
\begin{equation}\label{nonhomog}
x'(t)=L(t)x_t+z(t),\qquad t\geq0,
\end{equation}
where $z\colon[0,\infty)\rightarrow\mathbb R^d$ is a continuous function defined by
\[
z(t)=\begin{cases}
0\qquad&\text{for $t\in[0,s)$},\\
L(t)(\chi(t)u_t-u_t\chi_t)+\psi(t)u(t)\|u_t\|^{-1}
\qquad&\text{for $t\in[s,t_0)$},\\
0 \qquad&\text{for $t\in[t_0,\infty)$}.
\end{cases}
\]
Let $t\in[s,t_0)$ and $\theta\in[-r,0]$. Then
\[
(\chi(t)u_t-u_t\chi_t)(\theta)=u(t+\theta)\int_{t+\theta}^{t}\psi(\tau)\|u_\tau\|^{-1}\,d\tau
\qquad\text{whenever $t+\theta\geq s$},
\]
and
\[
(\chi(t)u_t-u_t\chi_t)(\theta)=u(t+\theta)\int_{s}^{t}\psi(\tau)\|u_\tau\|^{-1}\,d\tau
\qquad\text{whenever $t+\theta<s$}.
\]
In both cases, 
Claim~\ref{C} implies for $\tau\in[s,t]$,
\[
|u(t+\theta)|\leq\|u_{t}\|=\|T(t,\tau)u_\tau\|\leq Q\|u_\tau\|
\]
so that $|u(t+\theta)|\|u_\tau\|^{-1}\leq Q$. (We have used that $u_\tau=T(\tau,s)u_s\in\mathcal S(\tau)$ for $\tau\geq s$, a consequence of~\eqref{eq: invariance S U}.)
If $t+\theta\geq s$, then from the last inequality, we find that
\[
|u(t+\theta)|\int_{t+\theta}^{t}\psi(\tau)\|u_\tau\|^{-1}\,d\tau\leq Q\int_{t+\theta}^{t}\psi(\tau)\,d\tau\leq Qr.
\]
If  $t+\theta<s$, then $t<s-\theta\leq s+r$, which, together with the same inequality as before, implies that
\[
|u(t+\theta)|\int_{s}^{t}\psi(\tau)\|u_\tau\|^{-1}\,d\tau\leq Q\int_{s}^{s+r}\psi(\tau)\,d\tau\leq Qr.
\]
From the above estimates, we conclude that
\[
\|\chi(t)u_t-u_t\chi_t\|\leq Qr\qquad\text{for $t\in[s,t_0]$}.
\]
The last inequality and~\eqref{bc} imply that
\begin{equation}\label{zbound}
|z(t)|\leq MQr+1\qquad\text{for $t\geq0$.}
\end{equation}
Since $\chi$ is bounded on $[-r,\infty)$ and $u_s=\phi\in\mathcal S(s)$ implies that $u$ is also bounded on $[s-r,\infty)$, we have that $x$ is a bounded solution of~\eqref{nonhomog}. In addition, we have that $x_0=0\in\mathcal U$.
From this and~\eqref{zbound}, by the application of Claim~\ref{LEM2}, we conclude that
\[
|x(t)|\leq A(MQr+1)\qquad\text{for $t\geq-r$.}
\]
From the last inequality, taking into account that $\psi \equiv 1$  on $[s+\epsilon,t_0-r-\epsilon]$ and $\|u_\tau\|=\|T(\tau,s)u_s\|\leq Q\|u_s\|=Q\|\phi\|$ for $\tau\geq s$ (cf.~Claim~\ref{C}), we obtain 
\[
\begin{split}
A(MQr+1)\geq\|x_{t_0}\|&=
\sup_{\theta \in [-r, 0]}|u(t_0+\theta)|\int_{s}^{t_0+\theta}\psi (\tau)\|u_\tau\|^{-1}\, d\tau\\
&\ge \|u_{t_0}\|\int_{s+\epsilon}^{t_0-r-\epsilon}\|u_\tau\|^{-1}\, d\tau 
 \ge \|u_{t_0}\| \frac{t_0-r-s-2\epsilon}{Q\|\phi \|}\\
&\ge \frac{1}{eQ}(t_0-r-s-2\epsilon),
\end{split}
\]
the last inequality being a consequence of~\eqref{lowup}.
Letting $\epsilon \to 0$ in the last inequality, we conclude that 
\[
A(MQr+1)\geq\frac{1}{eQ}(t_0-r-s).
\]
This, together with the choice of~$t_0$ in~\eqref{low}, implies that
\[
N\leq t_0-s\leq eAQ(MQr+1)+r
\]
contradicting~\eqref{Nchoice}. Thus, \eqref{e} holds whenever $s\geq0$ and $\phi\in\mathcal S(s)$.

If $t\ge s\geq0$, then $t-s=kN+h$ for some $k\in \N_0$ and $h\in[0,N)$. From~\eqref{e} and Claim~\ref{C}, we obtain for $\phi \in \mathcal S(s)$, 
\[
\begin{split}
\|T(t,s)\phi\| &=\|T(s+kN+h, s)\phi \|
\le Q\|T(s+kN, s)\phi \|\\
&\le Qe^{-k}\|\phi \|\le eQe^{-\frac{t-s}{N}}\|\phi \|.
\end{split}
\]
Hence, the conclusion of the claim holds with $D=eQ$ and $\lambda=1/N$. 
\end{proof}

\begin{Claim}\label{C-U}
There exists $Q'>0$ such that 
\[
\|T(t,s)\phi \| \le Q'\|\phi \|, 
\]
for $0\leq t\leq s$ and $\phi \in \mathcal U(s)$.
\end{Claim}
\begin{proof}[Proof of Claim~\ref{C-U}]
Given $s\geq 0$ and $\phi \in \mathcal U(s)$, there exists $\bar \phi \in \mathcal{U}$ such that $\phi=T(s,0)\bar \phi$.  Let $u\colon [-r, \infty)\to \R^d$ be the solution of Eq.~\eqref{LDE} on $[0, \infty)$ such that $u_0=\bar \phi$.
Choose a continuously differentiable function
$\psi\colon [-r,\infty)\to [0,1]$ such that $\operatorname{supp}\psi\subset[-r,s+1]$,
$\psi \equiv 1$ on $[-r,s]$ and $|\psi'| \le 2$ on $[-r,\infty)$. Define $x\colon [-r, \infty)\to \R^d$ and $z\colon [0, \infty)\to \R^d$ by 
\[
x(t)=\psi(t) u(t),\qquad t\geq-r,
\]
and 
\[
z(t)=\psi'(t)u(t)+\psi(t)L(t)u_t-L(t)(\psi_t u_t),\qquad t\geq0.
\]
Then, it follows readily that $x$ is continuous on $[-r,\infty)$, $z$ is continuous on $[0, \infty)$, $x$ is differentiable on $[0, \infty)$ and~\eqref{adm} holds.
Moreover, since $x(t)=0$ for $t\geq s+1$, it follows that $\sup_{t\ge -r}|x(t)|<\infty$.
Furthermore, $x_0=u_0\in\mathcal U$. Finally, proceeding as in the proof of Claim~\ref{C}, it can be shown that 
\[
\sup_{t\ge 0}|z(t)|\leq  (2Ke^a+2MKe^{a(r+1)})\| \phi\|.
\]
Since $\psi\equiv1$ on $[-r,s]$, from conclusion~\eqref{cgt} of Claim~\ref{LEM2}, we conclude that 
\[
\sup_{-r\leq t\leq s}|u(t)| \le \sup_{t\ge -r}|x(t)|  \le A\sup_{t\ge 0}|z(t)| \le 2A(Ke^a+MKe^{a(r+1)})\| \phi\|.
\]
This implies that the conclusion of the Claim holds with 
\[
Q':= 2A(Ke^a+2MKe^{a(r+1)})>0.
\]
\end{proof}

\begin{Claim}\label{exp est unstable}
There exist $D', \lambda' >0$ such that 
\[
\|T(t,s)\phi \| \le D'e^{-\lambda' (s-t)}\|\phi\|,
\]
for $0\leq t\leq s$ and $\phi \in \mathcal U(s)$.
\end{Claim}

\begin{proof}[Proof of Claim~\ref{exp est unstable}]
We claim that if
\begin{equation}\label{NchoiceU}
N'>eAQ'(MQ'r+1)
\end{equation}
with~$A$ and~$Q'$ as in Claims~\ref{LEM2} and~\ref{C-U}, respectively, then for every $s\geq N'$ and $\phi\in\mathcal U(s)$, 
\begin{equation}\label{eU}
\|T(t,s)\phi \| \le \frac 1 e \|\phi \|\qquad\text{for all $t\in[0,s-N']$}.
\end{equation}
Suppose, for the sake of contradiction, that \eqref{NchoiceU} holds and there exist $s\geq N'$ and $\phi\in\mathcal U(s)$ such that
\begin{equation}\label{lowU}
\|T(t_0,s)\phi \|>\frac 1 e \|\phi \|\qquad\text{for some $t_0\in[0,s-N']$}.
\end{equation}
Since $\phi\in\mathcal U(s)$, there exists $\bar \phi\in \mathcal{U}$ such that $\phi =T(s,0)\bar \phi$.
Let~$u$ denote the unique solution of Eq.~\eqref{LDE} with $u_0=\bar \phi$ so that $u_s=T(s,0)u_0=\phi$. By Claim~\ref{lem: invertibility}, we have that $u_{t_0}=T(t_0,s)u_s=T(t_0,s)\phi$.
From~\eqref{lowU} and Claim~\ref{C-U}, it follows that
\begin{equation}\label{lowupU}
\frac 1 e \|\phi \|<\|u_{t_0}\|=\|T(t_0,\tau)u_\tau\|\leq Q'\|u_\tau\|\qquad \text{for $\tau\geq t_0$}.
\end{equation}
Choose a continuous function $\psi\colon \mathbb{R}\to [0,1]$  such that
\[
\operatorname{supp}\psi\subset[-r,s+1]\qquad\text{and}\qquad\text{$\psi\equiv1$ on $[t_0,s]$.}
\]
In view of~\eqref{lowupU}, $u_{t_0}$ is a nonzero element of~$\mathcal U(t_0)$. From this and Claim~\ref{lem: invertibility}, we conclude that $u_t=T(t,t_0)u_{t_0}\neq0$ for all $t\geq0$. Therefore, we can define a function
$x\colon[-r,\infty)\rightarrow\mathbb R^d$ by
\[
x(t)=\chi(t)u(t)\qquad\text{for $t\geq-r$},
\]
where $\chi\colon[-r,\infty)\rightarrow\mathbb[0,\infty)$  is given by
\[
\chi(t)=\begin{cases}
\displaystyle\int_{0}^{\infty}\psi(\tau)\|u_\tau\|^{-1}\, d\tau
\qquad&\text{for $t\in[-r,0)$}, \\[10pt]
\displaystyle\int_{t}^{\infty}\psi(\tau)\|u_\tau\|^{-1}\, d\tau
\qquad&\text{for $t\geq 0 $.}
\end{cases}
\]
Note that 
\[x_0=cu_0=c\bar\phi\in\mathcal U,\qquad\text{where $c=\displaystyle\int_0^{\infty}\psi(\tau)\|u_\tau\|^{-1}\,d\tau$.}
\] 
Since $x(t)=0$ for  $t\geq s+1$, we have that $\sup_{t\geq -r}|x(t)|<\infty $. Moreover, $x$ is differentiable on $[0, \infty)$ and~\eqref{adm} holds with $z\colon [0,+\infty)\to \mathbb{R}^d$ given by
\begin{equation}\label{eq: z aux}
    z(t)=L(t)\left(\chi(t)u_t-u_t\chi_t\right)-\psi(t)u(t)\|u_t\|^{-1},\qquad t\geq0.    
\end{equation}
Let $t\in[0,\infty)$ and $\theta\in[-r,0]$. Then
\[
(\chi(t)u_t-u_t\chi_t)(\theta)=-u(t+\theta)\int_{t+\theta}^{t}\psi(\tau)\|u_\tau\|^{-1}\,d\tau
\qquad\text{whenever $t+\theta\geq 0$},
\]
and
\[
(\chi(t)u_t-u_t\chi_t)(\theta)=-u(t+\theta)\int_{0}^{t}\psi(\tau)\|u_\tau\|^{-1}\,d\tau
\qquad\text{whenever $t+\theta<0$}.
\]
If $t+\theta\geq0$, then Claim~\ref{C-U} implies for $\tau\geq t+\theta$,
\[
|u(t+\theta)|\leq\|u_{t+\theta}\|=\|T(t+\theta,\tau)u_\tau\|\leq Q'\|u_\tau\|
\]
so that $|u(t+\theta)|\|u_\tau\|^{-1}\leq Q'$.
Therefore, if $t+\theta\geq 0$, then the last inequality yields
\[
|u(t+\theta)|\int_{t+\theta}^{t}\psi(\tau)\|u_\tau\|^{-1}\,d\tau\leq Q'\int_{t+\theta}^{t}\psi(\tau)\,d\tau\leq Q'r.
\]
If  $t+\theta<0$, then, using Claim~\ref{C-U} again, we obtain for $\tau\geq0$,
\[
|u(t+\theta)|\leq\|u_{0}\|=\|T(0,\tau)u_\tau\|\leq Q'\|u_\tau\|
\]
and hence $|u(t+\theta)|\|u_\tau\|^{-1}\leq Q'$. If $t+\theta<0$, then $t<-\theta\leq r$, which, together with the last inequality, yields
\[
|u(t+\theta)|\int_{0}^{t}\psi(\tau)\|u_\tau\|^{-1}\,d\tau\leq Q'\int_{0}^{r}\psi(\tau)\,d\tau\leq Q'r.
\]
From the above estimates, we find that
\[
\|\chi(t)u_t-u_t\chi_t\|\leq Q'r\qquad\text{for $t\geq0$}.
\]
From the last inequality,~\eqref{bc} and~\eqref{eq: z aux}, we obtain that
\begin{equation}\label{znewbound}
|z(t)|\leq MQ'r+1\qquad\text{for $t\geq0$.}
\end{equation}
By the application of Claim~\ref{LEM2}, we conclude that
\[
|x(t)|\leq A(MQ'r+1)\qquad\text{for $t\geq-r$.}
\]
From this, using \eqref{lowupU} and the facts that $\psi\equiv1$ on $[t_0,s]$ and $\|u_\tau\|\leq Q'\|u_s\|=Q'\|\phi\|$ for $0\leq\tau\leq s$ (cf.~Claim~\ref{C-U}), it follows that
\[
\begin{split}
A(MQ'r+1)& \ge \|x_{t_0}\| =\sup_{\theta \in [-r, 0]}|u(t_0+\theta)|\chi (t_0+\theta)\\
&\ge \|u_{t_0}\|\int_{t_0}^{s}\|u_\tau\|^{-1}\, d\tau  
 \ge \|u_{t_0}\| \frac{s-t_0}{Q'\|\phi \|}
> \frac{s-t_0}{eQ'}.
\end{split}
\]
Hence
\[
s-t_0<eAQ'(MQ'r+1)
\]
contradicting the fact that $s-t_0\geq N'>eAQ'(MQ'r+1)$ (cf.~\eqref{NchoiceU}). Thus, \eqref{eU} holds whenever $s\geq N'$ and $\phi\in\mathcal U(s)$.

If $0\leq t\leq s$, then $s-t=kN'+h$ for some $k\in \N_0$ and $h\in[0,N')$. From~\eqref{eU} and Claim~\ref{C-U}, we obtain for $\phi \in \mathcal U(s)$, 
\[
\begin{split}
\|T(t,s)\phi\| &=\|T(t, t+kN'+h)\phi \|\\
&=\|T(t, t+kN')T(t+kN',t+kN'+h)\phi \| \\
&\le e^{-k}\|T(t+kN',t+kN'+h)\phi \|\\
&\le Q'e^{-k}\|\phi \|\le eQ'e^{-\frac{s-t}{N'}}\|\phi \|.
\end{split}
\]
Thus, the conclusion of the Claim holds with $D'=eQ'$ and $\lambda'=1/N'$.
\end{proof}

For each $t\geq0$, let $P(t)$ denote  the projection of~$C$ onto $\mathcal S(t)$ along~$\mathcal U(t)$ associated with the decomposition \eqref{SPLIT2}.

\begin{Claim}\label{Claim: projectionsbounded}
The projections $P(t)$, $t\geq0$, are uniformly bounded, i.e.
\begin{equation*}
\sup_{t\geq 0}\|P(t)\|<\infty.
\end{equation*}
\end{Claim}

\begin{proof}[Proof of Claim~\ref{Claim: projectionsbounded}]
The desired conclusion follows from~\eqref{bg}, Claim~\ref{exp est stable} and Claim~\ref{exp est unstable} by reasoning as in the proof of~\cite[Lemma~4.2]{MRS}. \end{proof}

Now we can easily complete the proof of Theorem \ref{T2}.
Let $\phi \in C$ and $t\geq s\geq 0$.  
Then,
\[ \begin{split}
T(t,s)\phi &= T(t,s)P(s)\phi +T(t,s)(\Id-P(s))\phi.
\end{split}
\]
From this, using \eqref{eq: invariance S U}, we obtain
\[ \begin{split}
P(t)T(t,s)\phi &= P(t)\left( T(t,s)P(s)\phi +T(t,s)(\Id-P(s))\phi\right)\\
&=P(t)T(t,s)P(s)\phi=T(t,s)P(s)\phi.
\end{split}
\]
Thus, 
\eqref{pro} is satisfied. By Claim~\ref{lem: invertibility}, $T(t,s)\rvert_{\Ker P(s)} \colon \Ker P(s)\to \Ker P(t)$ is invertible. Furthermore, by Claim \ref{Claim: projectionsbounded}, the projections $P(t)$, $t\geq 0$, are uniformly bounded. Combining this fact with Claims \ref{exp est stable} and \ref{exp est unstable}, we conclude that \eqref{eq: def est stable} and \eqref{eq: def est unstable} are also satisfied. Thus, Eq.~\eqref{LDE} admits an exponential dichotomy.  
\end{proof}

\section*{Acknowledgements}
L.~Backes was partially supported by a CNPq-Brazil PQ fellowship under Grant No. 307633/2021-7.
D.~Dragi\v cevi\' c was supported in part by Croatian Science Foundation under the Project IP-2019-04-1239 and by the University of Rijeka under the Projects uniri-prirod-18-9 and uniri-prprirod-19-16. 
M.~Pituk was supported by the Hungarian National Research, Development and Innovation Office grant no.~K139346.


\begin{thebibliography}{99}

\bibitem{BD0}
L. Backes and D. Dragi\v cevi\' c, \emph{Shadowing for infinite dimensional dynamics and exponential trichotomies}, Proc. Roy. Soc. Edinburgh Sect.~A \textbf{151} (2021), no. 3, 863--884.

\bibitem{BDPS}
L. Backes, D. Dragi\v cevi\' c, M. Pituk, and L. Singh, \emph{Weighted shadowing for delay differential equations}, Arch. Math. (Basel) \textbf{119} (2022),  539--552.

\bibitem{BDV}
L. Barreira, D. Dragi\v cevi\' c and C. Valls, Admissibility and Hyperbolicity, SpringerBriefs Math., Springer, Cham (2018).


\bibitem{BV}
L. Barreira and C. Valls, \emph{Stability of delay equations}, Electron. J. Qual. Theory Differ. Equ. \textbf{45} (2022), no. 45, 1--24.

\bibitem{BerBrav}
L Berezansky and E. Braverman, \emph{On exponential stability of a linear delay differential equation with an oscillating coefficient}, Appl. Math. Lett. \textbf{22} (2009), 1833--1837.

\bibitem{BerDib}
L. Berezansky, J. Dibl\'\i k, Z. Svoboda, and Z.~\v Smarda, \emph{Uniform exponential stability of linear delayed integro-differential vector equations}, J.~Differential Equations \textbf{270} (2021), 573--595.



\bibitem{Pujals}
N. Bernardes Jr., P.R. Cirilo, U.B. Darji, A. Messaoudi, and E.R. Pujals, \emph{Expansivity and shadowing in linear dynamics}, J. Math. Anal. Appl. \textbf{461} (2018), 796--816.



\bibitem{Coppel} W. Coppel, Stability and Asymptotic Behavior of Differential Equations,  D.C. Heath, Boston (1965).

\bibitem{Dal} Ju.L. Dalecki\u{\i}  and M.G. Kre\u{\i}n, Stability  of Solutions of Differential Equations in Banach Space,  Translations of Mathematical Monographs Vol. 43, American Mathematical Society, Providence, R.I., 1974.

\bibitem{DibZaf} J.~Dibl\'\i k and A.~Zafer, \emph{On stability of linear delay differential equations under Perron's condition}, Abstr. Appl. Anal.,  Article ID 134072 (2011), 9 pages.


\bibitem{Hale} J. Hale, Theory of Functional Differential Equations, Springer, New York, 1977.



\bibitem{MasSchaf}
J.L. Massera and J.J. Sch\"affer, Linear Differential Equations and Function Spaces, Academic Press, New York, 1966.

\bibitem{MRS}
N. Van Minh, F. Rabiger, and R. Schnaubelt, \emph{Exponential stability, exponential expansiveness and exponential dichotomy of evolution equations on the half-line}, Integral Equations Operator Theory, \textbf{32} (1998), 332--353.

\bibitem{Kurz} J.~Kurzweil, \emph{Solutions of linear nonautonomous functional differential equations which are exponentially bounded for $t\to-\infty$}, J.~Differential Equations \textbf{11} (1972), 376--384.


\bibitem{Pecelli2}G. Pecelli, \emph{Functional differential equations: Dichotomies, perturbations, and admissibility}, J.~Differential Equations \textbf{16} (1974), 72--102.




\bibitem{Rud}
W. Rudin, Functional Analysis, Second Edition, McGraw--Hill Inc., New York, 1991.


\bibitem{Schaffer} J.J. Sch\"affer, \emph{Linear difference equations: Closedness of covariant sequences}, Math. Ann. \textbf{187} (1970), 69--76.

\bibitem{Taylor} A.E. Taylor, Introduction to Functional Analysis, John Wiley \& Sons, New York, 1958.

\end{thebibliography}
\end{document}